\documentclass[a4paper,12pt,reqno]{amsart}
\usepackage{amsmath}
\usepackage{amsthm}
\usepackage{amsfonts}
\usepackage{amssymb}
\usepackage{amscd}
\usepackage{bm}
\usepackage{color}
\usepackage{enumerate}
\usepackage{extarrows}
\usepackage{epstopdf}
\usepackage{etoolbox}
\usepackage{exscale}
\usepackage{epsfig}
\usepackage{graphicx}
\usepackage{latexsym}
\usepackage{mathrsfs}
\usepackage{pifont}
\usepackage{paralist}

\usepackage[toc,page]{appendix}

\appto\appendix{\addtocontents{toc}{\protect\setcounter{tocdepth}{1}}}

\usepackage{hyperref}
\renewcommand\eqref[1]{(\ref{#1})} 

\usepackage[a4paper, hmargin={2.7cm,2.7cm},vmargin={3.3cm,3.3cm}]{geometry}

\newtheorem{theorem}{Theorem}[section]
\newtheorem{proposition}[theorem]{Proposition}
\newtheorem{lemma}[theorem]{Lemma}
\newtheorem{corollary}[theorem]{Corollary}
\theoremstyle{definition}
\newtheorem{definition}[theorem]{Definition}

\newtheorem{remark}[theorem]{Remark}

\numberwithin{equation}{section}

\newcommand{\supp}{\operatorname{supp}}

\begin{document}

\title[Fourier multipliers for Hardy spaces on graded Lie groups]{Fourier multipliers for Hardy spaces on graded\\ Lie groups}

\author{Qing Hong}
\address{Qing Hong:
\endgraf
School of Mathematics and Statistics,
 Jiangxi Normal University
 \endgraf
 Nanchang, Jiangxi 330022,  China}
\email{qhong@mail.bnu.edu.cn}

\author{Guorong Hu}
\address{Guorong Hu:
\endgraf
School of Mathematics and Statistics,
 Jiangxi Normal University
 \endgraf Nanchang, Jiangxi 330022, China}
\email{hugr@mail.ustc.edu.cn}

\author{Michael Ruzhansky}
\address{Michael Ruzhansky:
\endgraf
Department of Mathematics
\endgraf
 Analysis, Logic and Discrete Mathematics
\endgraf
 Ghent University
\endgraf
 Krijgslaan 281, Building S8, B 9000 Ghent, Belgium
 \endgraf
and
\endgraf
School of Mathematical Sciences
\endgraf
 Queen Mary University of London
 \endgraf
Mile End Road, London E1 4NS, United Kingdom}
\email{Michael.Ruzhansky@UGent.be}


\date{\today}



\thanks{G. Hu and Q. Hong were supported by the NNSF of China (Grant Nos. 11901256 and 12001251)
and the NSF of Jiangxi Province (Grant Nos. 20192BAB211001 and 20202BAB211001).
M. Ruzhansky was supported by the FWO Odysseus 1 grant G.0H94.18N: Analysis and Partial Differential Equations}


\begin{abstract}
In this paper, we investigate the $H^p(G) \rightarrow L^p(G)$, $0< p \leq 1$, boundedness
of multiplier operators defined via group Fourier transform on a graded Lie group $G$,
where $H^p(G)$ is the Hardy space on $G$.
Our main result extends those obtained in [Colloq. Math. \textbf{165} (2021), 1--30],
where the $L^1(G)\rightarrow L^{1,\infty}(G)$ and $L^p(G) \rightarrow L^p(G)$,
$1< p <\infty$, boundedness of such Fourier multiplier operators were proved.
\end{abstract}

\maketitle

\tableofcontents

\section{Introduction}
\allowdisplaybreaks

Many problems in harmonic analysis and partial differential equations are related to the study of
Fourier or spectral multipliers for certain function spaces. We start by recalling the classical Mihlin multiplier theorem.
It  says that if
a function $\sigma (\xi)$ defined on $\mathbb{R}^n \backslash \{0\}$ has continuous derivatives up to
$(\lfloor n/2\rfloor +1)$-th order, and satisfies
\begin{align} \label{mih}
|\partial_{\xi}^{\alpha} \sigma (\xi)| \leq C_\alpha |\xi|^{-|\alpha|}
\end{align}
for all $\xi \in \mathbb{R}^n \backslash \{0\}$ and all multi-indices $\alpha \in \mathbb{N}_0^n$ with length
$|\alpha| \leq \lfloor n/2\rfloor +1$, then the Fourier multiplier operator $T_\sigma$ associated with $\sigma$, initially defined
for $f \in \mathcal{S}(\mathbb{R}^n)$ via
\begin{align*}
T_\sigma f = \mathcal{F}^{-1} (\sigma \widehat{f}),
\end{align*}
extends to a bounded operator on $L^p (\mathbb{R}^n)$ for all $1<p <\infty$.
H\"{o}rmander \cite{Hor} improved this result by showing that the regularity condition on $\sigma(\xi)$
could be allowed to be of fractional order. More precisely he proved that if
$\sigma \in \mathcal{S}'(\mathbb{R}^n)$ satisfies
\begin{align} \label{hor}
\sup_{t >0} \|\eta (\cdot) \sigma (t \cdot)\|_{W^{2}_{s}(\mathbb{R}^n)} <\infty
\end{align}
for some $s > n/2$, where $\eta$ is a function in $C_0^{\infty}(\mathbb{R}^n \backslash \{0\})$ such that $|\eta(\xi)| \geq c >0$
on some annulus $\{r_1 < |\xi| < r_2\}$,
then $T_\sigma$ extends to a bounded operator on $L^p (\mathbb{R}^n)$, for all $1<p <\infty$. Here $W^{2}_s(\mathbb{R}^n)$
denote the Sobolev spaces on $\mathbb{R}^n$.
It is well known that condition \eqref{hor} is weaker than condition \eqref{mih}.
Calder\'{o}n and Torchinsky \cite{CT} extended Mihlin and H\"{o}rmander's multiplier theorem to the case $0< p \leq 1$,
proving that if $\sigma$ satisfies \eqref{hor} for some $s> n(1/p-1/2)$,  then $T_\sigma$ is bounded
on the Hardy space $H^p(\mathbb{R}^n)$.

Multipliers for Lebesgue or Hardy spaces have also been studied extensively in the context of Lie groups.
For spectral multipliers on Lie groups associated to one (or several) opeartors such as a sub-Laplacian,
see, for example, \cite{HS, Christ, MaMe, HZ, A, MaMu1, MaMu2} and the references therein.
Note that the optimality of a Mihlin-H\"{o}rmander condition
in terms of the topological or homogeneous dimensions for spectral multipliers on stratified groups is a very difficult
problem which has so far only been solved in the case of $2$-step \cite{MaMu2,HZ, MS, MaMu1}.
Concerning Fourier multipliers on Lie groups, to our best knowledge, the first work was done by Coifman and Weiss
in \cite{CW1}, where they  studied the Fourier multipliers of $SU(2)$, see also \cite{CW2}. After that,
investigations of Fourier multipliers on compact Lie groups has been focused on the central multipliers
\cite{St, Vr, We1}, untill the appearance of the recent works of the third-named author and Wirth
\cite{RW1,RW2} and Fischer \cite{Fischer}.
The rest of the literature concerning Fourier multipliers on Lie groups is
restricted to the motion group \cite{Rubin} and to the Heisenberg group \cite{DM2, Lin, Bagchi}.

Recently, Fischer and the third-named author \cite{FR1} investigated Fourier multipliers on graded Lie groups.
One of their main results is the following Mihlin-type Fourier multiplier theorem for $L^p$ spaces
on graded Lie groups. (Basic concepts concerning graded Lie groups and representation theory, and the definition of difference operators
will be recalled in Section \ref{pre}.)

\medskip
\noindent
{\bf Theorem A} (see \cite[Theorem 1.1]{FR1}).
{\it Let $G$ be a graded Lie group with homogeneous dimension $Q$.
Let $\sigma =\{\sigma (\pi), \pi \in \widehat{G}\}$ be a measurable field of operators in $L^\infty(\widehat{G})$.
Assume that there exist a positive Rockland operator $\mathcal{R}$ (of homogeneous degree $\nu$)
and an integer $N > Q/2$ divisible by the dilation weights $v_1,\cdots, v_n$ (see Section \ref{pre} for their definition)
such that
\begin{equation} \label{imi1}
\sup_{\pi \in \widehat{G}} \big\| \Delta^{\alpha}\sigma \  \pi(\mathcal{R})^{\frac{[\alpha]}{\nu}}
\big\|_{\mathscr{L}(\mathcal{H}_\pi)} <\infty
\end{equation}
and
\begin{equation} \label{imi2}
\sup_{\pi \in \widehat{G}} \big\|\pi(\mathcal{R})^{\frac{[\alpha]}{\nu}} \ \Delta^{\alpha}\sigma
\big\|_{\mathscr{L}(\mathcal{H}_\pi)} <\infty,
\end{equation}
hold for all $\alpha \in \mathbb{N}_0^n$ with $[\alpha] \leq N$.
Then the Fourier multiplier operator $T_{\sigma}$ defined via
\begin{equation} \label{Fouriermultiplier}
\mathcal{F}_G (T_\sigma f)(\pi) = \sigma(\pi)\widehat{f}(\pi)
\end{equation}
is of weak type
$(1,1)$, and is bounded on $L^p (G)$ for all $1 < p <\infty$.}

\medskip

Examining the proof of Theorem A (given in \cite{FR1}), we find that the condition \eqref{imi1} is sufficient
to give the weak $(1,1)$ estimate of $T_\sigma$. The latter along with the $L^2(G)$-boundedness of $T_\sigma$ (which follows
from the Plancherel theorem) and an interpolation argument yields the $L^p(G)$-boundedness of $T_\sigma$ for $1 < p \leq 2$.
Note that $T^\ast_\sigma = T_{\sigma^\ast}$, where $\sigma^\ast = \{\sigma(\pi)^\ast, \pi \in \widehat{G}\}$, and
that if $\sigma$ satisfies \eqref{imi2} then $\sigma^\ast$ satisfies \eqref{imi1}.
Hence, if $\sigma$ satisfies \eqref{imi1} and \eqref{imi2}, then both $T_\sigma$ and $T_\sigma^\ast$
are bounded on $L^p(G)$ for $1 < p \leq 2$, which implies that $T_\sigma$ is bounded on $L^p(G)$ for all $1 < p < \infty$.

The purpose of the present paper is to extend Theorem A to the case $0< p \leq 1$ by investigating
the $H^p(G) \rightarrow L^p(G)$ boundedness of $T_\sigma$, where $H^p(G)$
is the Hardy space on $G$. Our main results is the following

\begin{theorem} \label{main1}
Let $G$ be a graded Lie group with homogeneous dimension $Q$.
Let $\sigma =\{\sigma (\pi), \pi \in \widehat{G}\}$ be a measurable field of operators in $L^\infty(\widehat{G})$.
Let $0< p \leq 1$.
Assume that there exist a positive Rockland operator $\mathcal{R}$ (of homogeneous degree $\nu$)
and an integer $N > Q(1/p-1/2)$ divisible by the dilation weights $v_1,\cdots, v_n$
such that
\begin{equation*}
\sup_{\pi \in \widehat{G}} \big\| \Delta^{\alpha}\sigma \  \pi(\mathcal{R})^{\frac{[\alpha]}{\nu}}
\big\|_{\mathscr{L}(\mathcal{H}_\pi)} <\infty
\end{equation*}
holds for all $\alpha \in \mathbb{N}_0^n$ with $[\alpha] \leq N$,
Then the Fourier multiplier operator $T_{\sigma}$ defined by \eqref{Fouriermultiplier} is bounded from $H^p (G)$
to $L^p(G)$.
\end{theorem}

Some remarks concerning Theorem \ref{main1} are in order.
\begin{itemize}
\item[(1)] Taking $p =1$ in Theorem \ref{main1}, we have the $H^1(G) \rightarrow L^1(G)$ boundedness of $T_\sigma$,
under the assumption that $\sigma$ satisfies \eqref{imi1} for some integer $N > Q/2$ which is divisible by the dilation
weights $v_1, \cdots, v_n$. Thus (by interpolation) our result also implies the $L^p(G)$-boundedness of $T_\sigma$
stated in Theorem A under the same assumptions.
\item[(2)] In the abelian Euclidean setting, that is, $(\mathbb{R}^n, +)$ with the usual isotropic
dilation with $\mathcal{R}$ being the Laplace operator,
\eqref{imi1} is equivalent to \eqref{imi2}, and each of them reduces to
\eqref{mih}. Indeed,  the Euclidean abelian setting, all the dilations weights $v_1, \cdots, v_n$
are equal to $1$,
and $\pi(\mathcal{R})$ reduces to $|\xi|^2$, where $\xi$ is the
(Fourier) dual variable.
\item[(3)] As we mentioned before, the optimality of the Mihlin-H\"{o}rmander condition for multipliers on Lie groups
is a very deep problem. It is known that on any 2-step stratified group the sufficient and necessary condition
for $L^p$-boundedness of a spectral multiplier $F(\mathcal{\mathcal{L}})$ (where $\mathcal{L}$ is a sub-Laplaican)
is that $F$ satisfies a scale-invariant smoothness condition of order $s > n/2$, where $n$ is the topological dimension
of the group (see \cite{MaMu2}). It is natural to ask whether the condition $N > Q (1/ p -1/2)$ in Theorem \ref{main1} can
be replaced by $N > n(1/p -1/2)$. However, we do not indent to study this problem in the present paper.
\end{itemize}

To prove Theorem \ref{main1} we shall mainly use an atomic decomposition of $H^p(G)$, the Littlewood--Paley decomposition, and a
Taylor formula with integral remainder on homogeneous groups which is due to Bonfiglioli \cite{Bon}.
Hulanicki's theorem will also play an important role in our proof.

 \medskip
This paper is organized as follows. In Section \ref{pre}, we recall basic notions concerning graded Lie groups,
basic representation theory, the group Fourier transform, Rockland operators and
difference operators. In Section \ref{hardyspaces},
we recall some basic facts about Hardy spaces on graded Lie groups, including their atomic decomposition.
The proof of our main theorem will
be given in Section \ref{proofoftheorem}.

 \medskip
{\it Notation.} We use
 $\mathbb{N}_{0}$ to denote the set of all nonnegative integers.
For a nonnegative number $s$, we denote by $\lfloor s \rfloor$ the largest integer less than or equal to $s$.
If $\mathcal{H}_1$ and $\mathcal{H}_2$ are two Hilbert spaces, we denote by $\mathscr{L}(\mathcal{H}_1, \mathcal{H}_2)$
the Banach space of the bounded linear operators from $\mathcal{H}_1$ to $\mathcal{H}_2$.
When $\mathcal{H}_1 = \mathcal{H}_2 = \mathcal{H}$ then we write $\mathscr{L} (\mathcal{H}_1, \mathcal{H}_2) = \mathscr{L}(\mathcal{H})$.
The letter $C$ will denote positive constants, which are independent of the main
variables involved and whose value may vary at every occurrence.
By writing $f \lesssim g$ we mean that $f \leq Cg$. If $f \lesssim g$ and $g \lesssim f$, we also write $f \sim g$.

\section{Preliminaries} \label{pre}

\subsection{Graded Lie groups and their homogeneous structure}
A Lie group $G$ is said to be graded if it is connected and simply connected, and
its Lie algebra $\mathfrak{g}$ is endowed with a vector space decomposition
$\mathfrak{g}=\oplus_{k =1}^{\infty} \mathfrak{g}_k$
(where all but finitely many of the $\mathfrak{g}_{k}$'s are $\{0\}$)
such that $[\mathfrak{g}_{k}, \mathfrak{g}_{k'}]\subset \mathfrak{g}_{k + k'}$ for all
$k, k' \in \mathbb{N}$. Such a group is necessarily nilpotent, and
the exponential map $\exp: \mathfrak{g} \rightarrow  G$ is a diffeomorphism.
Examples of graded Lie groups include the Euclidean space $\mathbb{R}^n$, the Heisenberg group $\mathbb{H}^n$ and,
more generally, all stratified Lie groups.

We choose and fix a basis $\{X_1, \cdots, X_n\}$ of $\mathfrak{g}$, so that it is adapted to the gradation, i.e.,
$\{X_1, \cdots, X_{n_1}\}$ (possibly $\emptyset$) is a basis of $\mathfrak{g}_1$, $\{X_{n_1 +1}, \cdots, X_{n_1 + n_2}\}$
(possibly $\emptyset$) is a basis of $\mathfrak{g}_2$, and so on.
Via the map
\begin{align} \label{iu}
\mathbb{R}^n \ni (x_1, \cdots, x_n) \mapsto \exp(x_1 X_1 + \cdots + x_n X_n) \equiv x\in G,
\end{align}
each point $(x_1, \cdots, x_n) \in \mathbb{R}^n$ is identified with the point $x \in G$.
This map
takes the Lebesgue measure on $\mathbb{R}^n$ to a bi-invariant Haar measure $\mu$ on $G$.
We denote the group identity of $G$ by $e$.

The Lie algebra $\mathfrak{g}$ is equipped with a natural family of dilations $\{\delta_r\}_{r>0}$ which are linear
mappings from $\mathfrak{g}$ to $\mathfrak{g}$ determined by
\begin{align*}
\delta_r X= r^k X \quad \mbox{  for } X \in \mathfrak{g}_k
\end{align*}
For each $j  \in \{1, \cdots, n\}$, let $v_j$ be the unique positive integer such that $X_j \in \mathfrak{g}_{v_j}$.
Then we have $\delta_r X_j = r^{v_j} X_j$, $j =1, \cdots, n$. The associated group dilation is given by
\begin{align*}
\delta_r x  = (r^{v_1}x_1, \cdots, r^{v_n} x_n),
\end{align*}
for $x = (x_1, \cdots, x_n) \in G$ and $r>0$.
The integers $v_1, \cdots, v_n$ are referred to as weights
of the dilations $\{\delta_t\}_{t>0}$, and the positive integer
\begin{align*}
Q:= \sum_{k=1}^\infty k (\dim \mathfrak{g}_k )=\sum_{j=1}^n v_j
\end{align*}
 is called the homogeneous dimension of $G$.

A homogeneous quasi-norm on G is a continuous function $x\rightarrow |x|$ from $G$ to $[0, \infty)$
which vanishes only at $e$ and satisfies that $|x^{-1}| = |x|$ and $|\delta_r x | = r|x|$ for all
$x \in G$ and $r > 0$. An example of homogeneous quasi-norm on $G$ is given by
\begin{align} \label{xkappa}
|x|_\kappa = \left(\sum_{j=1}^n x_j^{2\kappa /v_j} \right)^{1/(2\kappa)},
\end{align}
where $\kappa$ is the smallest common multiple to the weights $v_1, \cdots, v_n$.
Any two homogeneous quasi-norms on $G$ are equivalent (see \cite{FS}).
Henceforth we fix a homogenous quasi-norm $|\cdot|$ on $G$. It satisfies a quasi-triangle inequality: there exists a constant
$\gamma \geq 1$ such that
\begin{align} \label{quas}
 |xy| \leq \gamma (|x| + |y|)
 \end{align}
for all $x,y \in G$.

There is an analogue of polar coordinates on homogeneous groups with the homogeneous dimension $Q$ replacing
the topological dimension $n$, see \cite{FS}:
\begin{align*}
\forall f \in L^1 (G) \quad \int_G f(x)d\mu(x) = \int_0^{\infty} \int_{\mathfrak{S}} f(\delta_r y)r^{Q-1}d\sigma (y)dr,
\end{align*}
where $d\sigma$ is a (unique) positive Borel measure on the unit sphere $\mathfrak{S} := \{x \in G: |x| =1\}$.
This implies that for $0<r < R < \infty$ and $\theta \in \mathbb{R}$,
\begin{align} \label{integ}
\int_{r \leq |x| \leq R} |x|^{\theta -Q}d\mu(x) =
\begin{cases}
C \theta^{-1} (R^\theta - r^\theta)  &\mbox{if } \theta \neq 0, \\
C \log(R /r) &\mbox{if } \theta =0.
\end{cases}
\end{align}
Consequently, if $\theta>0$ then $|\cdot|^{\theta -Q}$ is integrable near the group identity $e$, and
if $\theta <0$ then $|\cdot|^{\theta - Q}$ is integrable near $\infty$.

Since $G$ has been identified with $\mathbb{R}^n$ via the map given in \eqref{iu}, functions on $G$
can be viewed as functions on $\mathbb{R}^n$, and vise versa. This leads naturally to the notions of
test function classes $\mathcal{D}(G)$, $\mathcal{S}(G)$ and the distribution spaces $\mathcal{D}'(G)$,
$\mathcal{S}'(G)$. For example, a function $f$ is
said to be in the Schwartz class $\mathcal{S}(G)$ if $f \circ \exp$ is a Schwartz function
on $\mathbb{R}^n$. The coordinate function $G \ni x =(x_1, \cdots, x_n)
\mapsto x_1 \in \mathbb{R}$ is denoted by $x_1$.
 For a multi-index $\alpha =(\alpha_1, \cdots, \alpha_n) \in \mathbb{N}_0^n$, we define $x^{\alpha} = x_1^{\alpha_1}
\cdots x_n^{\alpha_n}$, as a function on $G$. Similarly, we set $X^{\alpha} = X_1^{\alpha_1} \cdots X_n^{\alpha_n}$
in the universal enveloping Lie algebra $\mathfrak{U}(\mathfrak{g})$ of $\mathfrak{g}$.
We shall follow the usual custom of identifying
each vector of $\mathfrak{g}$ with a left-invariant vector field on $G$ and, more generally, we identify
the universal enveloping Lie algebra of $\mathfrak{g}$ with the left-invariant differential
operators. In what follows we keep the same notation for the vectors and the corresponding operators.
By the Poincar\'{e}-Birkhoff-Witt theorem, the $X^\alpha$'s form a basis for the algebra of the left-invariant
differential operators on $G$.

In a canonical way the dilations $\{\delta_r\}_{r>0}$ lead to the notions of homogeneity for functions and operators. For instance
the degree of homogeneity of the function $x^\alpha$ and the differential operator $X^\alpha$ is
\begin{align*}
[\alpha] : = \sum_{j =1}^n v_j \alpha_j.
\end{align*}

A function $P: G \rightarrow \mathbb{C}$ is called a polynomial, if it is of the form
\begin{align*}
P(x) =\sum_{\alpha \in \mathbb{N}_0^n} c_\alpha x^\alpha
\end{align*}
where all but finitely many of the complex coefficients $c_\alpha$ vanish. The homogeneous degree
of the polynomial $P$ is defined as $\max \{[\alpha]: c_\alpha \neq 0 \}$. For $M \in \mathbb{N}_0$, we set
\begin{align*}
\mathcal{P}_M := \{\mbox{all polynomials on }G \mbox{ with homogeneous degree } \leq M\}.\hat{}
\end{align*}

We denote by $\widetilde{X}_1, \cdots, \widetilde{X}_n$ the corresponding basis for right-invariant vector fields, that is,
\begin{align*}
\widetilde{X}_j f (x) = \frac{d}{dt} f \big(\exp(tX_j) x\big)\big|_{t=0}, \quad j =1, \cdots, n.
\end{align*}
Also, for $\alpha \in \mathbb{N}_0^n$, we set $\widetilde{X}^\alpha = \widetilde{X}_1^{\alpha_1} \cdots \widetilde{X}_n^{\alpha_n}$.

If $f$ and $g$ are measurable functions on $G$, then their convulution is defined by
\begin{align*}
f \ast g (x) = \int_G f(y)g(y^{-1}x)d\mu(y) = \int_G f(xy^{-1})g(y)d\mu(y),
\end{align*}
provided that the integrals converge. For any multi-index $\alpha \in \mathbb{N}_0^n$ and
sufficiently good functions $f$ and $g$, we have  (see \cite[Chapter 1]{FS})
\begin{align} \label{decon}
X^{\alpha} (f \ast g) = f \ast (X^\alpha g), \quad \widetilde{X}^\alpha (f \ast g) = (\widetilde{X}^{\alpha} f) \ast g,
\quad (X^\alpha f) \ast g = f \ast (\widetilde{X}^\alpha g).
\end{align}

\subsection{Fourier analysis on graded Lie groups} The general theory of representation
of Lie groups may be found in \cite{Dix}. Here we also refer to
\cite{FR} for a description which is more adapted to our particular context.

A representation $\pi$ of a Lie group $G$ on a Hilbert space $\mathcal{H}_{\pi} \neq \{0\}$
is a homomorphism from $G$ into the group of bounded linear operators on $\mathcal{H}_{\pi}$ with bounded inverse.
More precisely,

\begin{itemize}
\item for every $x \in G$, the linear mapping $\pi(x) : \mathcal{H}_{\pi} \rightarrow \mathcal{H}_{\pi}$ is bounded and has bounded inverse;
\item for every $x, y \in G$, we have $\pi(xy )=\pi(x) \pi(y)$.
\end{itemize}

 A representation $\pi$ of $G$ is called irreducible if it has no closed invariant subspaces.
$\pi$ is called unitary if $\pi(x)$ is unitary for every $x \in G$, and is called strongly continuous
if the mapping $\pi: G \rightarrow \mathscr{L}(\mathcal{H_{\pi}})$ is continuous with respect to the strong operator
topology in $\mathscr{L}(\mathcal{H}_{\pi})$. Two representations $\pi_1$ and $\pi_2$ are said to be equivalent if there exists a
bounded linear mapping $A: \mathcal{H}_{\pi_1} \rightarrow \mathcal{H}_{\pi_2}$ between their representation spaces with
a bounded inverse such that the relation
$A \pi_1 (x) = \pi_2 (x)A$ holds for all $x \in G$. In this case we write
$\pi_1 \sim \pi_2$, and denote their equivalence class by $[\pi_1] = [\pi_2]$.
The set of all equivalence classes of strongly continuous irreducible unitary
representations of $G$ is called the unitary dual of $G$ and is denoted
by $\widehat{G}$. In what follows, we will identify one representation $\pi$ with
its equivalent class $[\pi]$.

For a unitary representation of $G$, the corresponding infinitesimal representation which acts on
the universal enveloping algebra $\mathfrak{U}(\mathfrak{g})$ of the Lie algebra $\mathfrak{g}$ is still
denoted by $\pi$. This is characterized by its action on $\mathfrak{g}$:
\begin{align*}
\pi (X) = \partial_{t =0} \pi (e^{tX}), \quad X \in \mathfrak{g}.
\end{align*}
The infinitesimal action acts on the space $\mathcal{H}_{\pi}^{\infty}$ of smooth vectors, that is, the space of vectors
$v \in \mathcal{H}_{\pi}$ such that the function $G \ni x \mapsto \pi(x) v \in \mathcal{H}_{\pi}$ is of class $C^{\infty}$.

The Fourier coefficients or group Fourier transform of a function $f \in L^1 (G)$ at $\pi \in \widehat{G}$ is defined by
\begin{align*}
\mathcal{F}_G f (\pi) \equiv \widehat{f}(\pi) \equiv \pi (f) := \int_G f(x) \pi(x)^\ast d\mu(x).
\end{align*}
It is readily seen that
\begin{align*}
\|\widehat{f}(\pi)\|_{\mathscr{L}(\mathcal{H}_\pi)} \leq \|f\|_{L^1 (G)}.
\end{align*}
For $f, g \in L^1 (G)$, we also have
\begin{align*}
\widehat{f \ast g} (\pi) = \widehat{g}(\pi) \widehat{f}(\pi).
\end{align*}

There exists a unique positive Borel measure $\widehat{\mu}$ on $\widehat{G}$, called the Plancherel measure,
such that for any continuous function $f$ on $G$ with compact support, one has
\begin{align*}
\int_G |f(x)|^2 d\mu(x) = \int_{\widehat{G}} \|\mathcal{F}_G f (\pi)\|_{HS(\mathcal{H}_{\pi})}^2 d\widehat{\mu}(\pi),
\end{align*}
where $\|\cdot\|_{HS(\mathcal{H}_{\pi})}$ denotes the Hilbert-Schmidt norm on the space $HS(\mathcal{H}_\pi) \sim \mathcal{H}_{\pi}
\otimes \mathcal{H}_{\pi}^\ast$  of Hilbert-Schmidt operators on the Hilbert space $\mathcal{H}_{\pi}$.
Since $L^1(G) \cap L^2(G)$ is dense in $L^2(G)$, the Fourier transform $\mathcal{F}_G$ extends
to a unitary operator from $L^2(G)$ onto $L^2 (\widehat{G})$.

By the general theory on locally compact unimodular groups of type I (see e.g. \cite{Dix}), if $T$ is an $L^2$-bounded
operator on $G$ which commutes with left-translations, then  there exists a field of bounded operators $\widehat{T}(\pi)$ such that
for all $f \in L^2 (G)$,
\begin{align*}
\mathcal{F}_G (Tf)(\pi) = \widehat{T} (\pi) \widehat{f}(\pi) \quad \mbox{a.e. } \pi \in \widehat{G}.
\end{align*}
Moreover, we have $$\|T\|_{\mathscr{L}(L^2(G))} = \sup_{\pi \in \widehat{G}}
\|\widehat{T}(\pi)\|_{\mathscr{L}(\mathcal{H}_\pi)},$$
where the supremum here is understood as the essential supremum with respect to the Plancherel
measure $\mu$. Conversely, given any $\sigma =\{\sigma (\pi), \pi \in \widehat{G}\} \in L^\infty (\widehat{G})$,
there is a corresponding operator $T_\sigma$ given by
\begin{align*}
\mathcal{F}_G (T_\sigma f) (\pi) = \sigma(\pi) \widehat{f}(\pi), \quad f \in L^2 (G).
\end{align*}
By the Plancherel theorem, $T_\sigma$ is bounded on $L^2(G)$ with $\|T_\sigma\|_{\mathscr{L}(L^2(G))} =\|\sigma\|_{L^\infty (\widehat{G})}$.

If $\pi$ is a unitary irreducible representation of $G$ and $r >0$, we define $r \cdot \pi$
to be the unitary irreducible representation such that
\begin{align*}
r \cdot \pi (x) = \pi (\delta_r x), \quad x \in G.
\end{align*}

\subsection{Rockland operators}
Let $G$ be a graded Lie group. A left-invariant differential operator $\mathcal{R}$
on $G$ is called a Rockland operator if it is homogeneous of positive degree and
for each unitary irreducible non-trivial representation $\pi$ of $G$,
the operator $\pi(\mathcal{R})$ is injective on $\mathcal{H}_{\pi}^{\infty}$.
Rockland operators may be defined on any homogeneous group,
however it turns out that the existence of a Rockland operator on a homogeneous group
implies that (the Lie algebra of) the group admits a gradation.
This is the reason why we and the authors in \cite{FR1} consider the setting of graded
Lie groups. On any graded Lie group $G$, the operator
\begin{align*}
\sum_{j =1}^n (-1)^{\frac{\nu_0}{v_j}} c_j X_j^{2 \frac{\nu_0}{v_j}}
\end{align*}
with $c_j >0$ is a Rockland operator of homogeneous degree 2$\nu_0$ if $\nu_0$
is any common multiple of $v_1, \cdots, v_n$.

We will mainly consider positive Rockland operators.
A Rockland operator $\mathcal{R}$ is said to be positive, if
\begin{align*}
\int_G \mathcal{R}f (x)\overline{f(x)}d\mu(x) \geq 0
\end{align*}
for all $f \in \mathcal{S}(G)$. If a Rockland operator $\mathcal{R}$ is positive
then $\mathcal{R}$ and $\pi (\mathcal{R})$ admit self-adjoint extensions on
$L^2 (G)$ and $\mathcal{H}_{\pi}$, respectively. We use the same notation for
their self-adjoint extensions. By the spectral theory, we have
\begin{align*}
\mathcal{R} = \int_0^{\infty} \lambda dE_{\mathcal{R}}(\lambda)\quad \mbox{and}
\quad \pi(\mathcal{R}) = \int_0^{\infty} \lambda dE_{\pi(\mathcal{R})}(\lambda),
\end{align*}
where $E_{\mathcal{R}}(\lambda)$ (resp. $E_{\pi(\mathcal{R})} (\lambda)$) is the
resolution of the identity associated to $\mathcal{R}$ (resp. $\pi(\mathcal{R})$).

For any bounded Borel function $\varphi$ on $[0,\infty)$, the operator
\begin{equation*}
\varphi(\mathcal{R}) = \int_{0}^{\infty} \varphi(\lambda)dE_{\mathcal{R}}(\lambda)
\end{equation*}
is bounded on $L^{2}(G)$, and commutes with left translations. Thus, by the Schwartz
kernel theorem, there exists a distribution $K_{\varphi(\mathcal{R})} \in \mathcal{S}'(G)$ such that
\begin{equation*}
\varphi(\mathcal{R})f =f \ast K_{\varphi(\mathcal{R})}, \quad \forall f \in \mathcal{S}(G).
\end{equation*}
Note that the point $\lambda =0$ may be neglected in the spectral resolution, since the projection measure of $\{0\}$ is zero
(see \cite{HJL} or \cite[Remark 4.2.8]{FR}). Consequently we should regard $\varphi$ as a function on  $(0, \infty)$ rather than on $[0,\infty)$.

We now recall Hulanicki's theorem, which will play an important role in the proof of our main result.

\begin{theorem}[Hulanicki \cite{Hulanicki}]
Let $G$ be a graded Lie groups and let $\mathcal{R}$ be a positive Rockland
operator on $G$.
For any $M_1 \in \mathbb{N}$ and $M_2 \geq 0$, there exist $C =C(M_1,M_2) >0$ and $k = k(M_1, M_2), k'=k'(M_1,M_2) \in \mathbb{N}_0$
such that, for any $\varphi \in C^k (0, \infty)$, the convolution kernel $K_{\varphi(\mathcal{R})}$ of $\varphi(\mathcal{R})$ satisfies
\begin{align*}
\sum_{[\alpha] \leq M_1} \int_G |X^\alpha K_{\varphi(\mathcal{R})} (x)|(1+|x|_\kappa)^{M_2}d\mu(x) \leq C \sup_{\substack{\lambda\in (0,\infty) \\  \ell \in \{0,1,\cdots, k\} \\
\ell' \in \{0,1,\cdots, k'\}}}
 (1+\lambda)^{\ell'}\left|\frac{d^\ell}{d\lambda^\ell}\varphi(\lambda) \right|.
\end{align*}

The same result with the right-invariant vector fields $\widetilde{X}_j$'s instead of the
left-invariant vector fields $X_j$'s also holds.
\end{theorem}

\begin{corollary} \label{corohul}
Let $\mathcal{R}$ be a positive Rockland operator on a stratified Lie group $G$.
If $\varphi$ is a function on $(0, \infty)$ such that $\varphi = \widetilde{\varphi}|_{(0,\infty)}$ for some $\widetilde{\varphi}
\in \mathcal{S}(\mathbb{R})$, then $K_{\varphi(\mathcal{R})} \in \mathcal{S}(G)$.
\end{corollary}

\subsection{Difference operators}
The Mihlin condition \eqref{mih} is formulated in terms of the derivatives with respect to the Fourier variable $\xi$.
However, for a field $\sigma =\{\sigma (\pi), \pi \in \widehat{G}\}$ of operators, there is no direct way
to define an analogue of derivatives with respect to the Fourier variable $\pi$. To generalize the symbolic conditions to
the setting of graded Lie groups, Fischer and Ruzhansky \cite{FR} introduced the so-called difference operators,
whose definition we now recall.

For $a, b \in \mathbb{R}$, we denote by $\mathscr{L}_L (L_a^2(G), L_b^2(G))$ the subspace of all $T \in \mathscr{L}(L^2_a(G), L^2_b(G))$
which are left-invariant. Here $L^2_a(G)$ is the Bessel potential space (fractional Sobolev space) defined in \cite{FR2}.
Define
\begin{align*}
\mathcal{K}_{a,b} (G):= \Big\{K \in \mathcal{S}'(G): \   &\mbox{the operator } \mathcal{S}(G) \ni
f \mapsto f \ast K  \mbox{ extends to} \\
&\mbox{a bounded operator from }
 L_a^2(G) \mbox{ to } L_b^2(G)) \Big\}
\end{align*}
and define $L_{a,b}^{\infty}(\widehat{G})$
to be the space of all fields $\sigma = \{\sigma(\pi), \pi \in \widehat{G}\}$ such that
\begin{align*}
\|\sigma\|_{L^{\infty}_{a,b}(\widehat{G})} := \sup_{\pi \in \widehat{G}} \|\pi(I +\mathcal{R})^{\frac{b}{\nu}}
\sigma(\pi)\pi(I +\mathcal{R})^{-\frac{a}{\nu}} \|_{\mathscr{L}(\mathcal{H}_{\pi})} <\infty.
\end{align*}
From \cite[Proposition 5.1.24]{FR} we see that, if $\sigma \in L^{\infty}_{a,b}(\widehat{G})$
then the Fourier multiplier operator $T_{\sigma}$ corresponding to $\sigma$ belongs to $\mathscr{L}_L (L_a^2 (G), L_b^2(G))$ with
\begin{align*}
 \|T_{\sigma}\|_{\mathscr{L}(L_a^2 (G), L_b^2 (G))} = \|\sigma\|_{L^{\infty}_{a,b}(\widehat{G})}.
\end{align*}
Conversely, if $T \in \mathscr{L}_L (L_a^2(G), L_b^2(G))$, then there exists a unique $\sigma \in L^{\infty}_{a,b}(\widehat{G})$
such that
\begin{align*}
\mathcal{F}_G (T f)(\pi) = \sigma(\pi) \widehat{f}(\pi), \quad f \in L^2(G).
\end{align*}
In this case, denoting by $K \in \mathcal{K}_{a,b}(G)$ the convolution kernel of $T$, we define
\begin{align*}
\mathcal{F}_G K = \sigma \quad \mbox{and} \quad \mathcal{F}_G^{-1}\sigma = K.
\end{align*}
This extends the definition of Fourier transform to the space
$\mathcal{K}_{a,b} (G)$. See \cite[Definition 5.1.25]{FR}.

 For $\alpha \in \mathbb{N}_0^n$ and $\sigma  =\{\sigma (\pi), \pi \in \widehat{G}\}\in L^{\infty}_{a,b}(\widehat{G})$,
 the difference operator $\Delta^\alpha$ acting on $\sigma$ is defined according to the formula (see
 \cite[Definition 5.2.1]{FR})
\begin{align*}
\Delta^\alpha \sigma (\pi) = \mathcal{F}_G (q_\alpha \mathcal{F}_G^{-1}\sigma)(\pi) \quad \mbox{for a.e. } \pi \in \widehat{G},
 \end{align*}
 where $q_\alpha(x) =x^\alpha$.
Analogously to the derivatives in the Euclidean setting,
the operator $\Delta^{\alpha}$ satisfies the Lebnitz rule \cite[Section 5.2.2]{FR}:
\begin{align} \label{leibniz}
\Delta^{\alpha} (\sigma \tau) =\sum_{\alpha_1 +\alpha_2 =\alpha} C_{\alpha_1, \alpha_2} \Delta^{\alpha_1} (\sigma)
\Delta^{\alpha_2} (\tau), \quad \sigma, \tau \in L^{\infty}_{a,b}(\widehat{G}).
\end{align}

\section{Hardy spaces on graded Lie groups} \label{hardyspaces}

A comprehensive theory of Hardy spaces on general homogeneous groups was built
by Folland and Stein \cite{FS}. Since all graded Lie groups are homogeneous, the theory
in \cite{FS} applies to our setting.

In what follows, $G$ is always a graded Lie group with homogeneous dimension $Q$.

\subsection{Definition of Hardy spaces on $G$} We first introduce some maximal functions.
Given a function $\Phi \in \mathcal{S}(G)$, we defined the nontangential maximal function
$M_\Phi f$ and the radial maximal function $M_\Phi^0 f$ of $f \in \mathcal{S}'(G)$ by
\begin{align*}
M_\Phi f (x) : =  \sup_{|x^{-1}y| <t} |f \ast \Phi_t (y)|
\quad \mbox{and} \quad
M_\Phi^0 f (x): = \sup_{t >0 } |f \ast \Phi_t (x)|,
\end{align*}
respectively,
where $\Phi_t (x):= t^{-Q} \Phi(\delta_{t^{-1}}x)$. We then define the grand maximal function $M_{(N)}f$ for each $N\in\mathbb{N}$ by
\begin{align*}
M_{(N)} f(x) := \sup_{\Phi \in \mathcal{S}(G), \|\Phi\|_{(N)} \leq 1} M_\Phi f(x),
\end{align*}
where
\begin{align*}
\|\Phi\|_{(N)}: = \sup_{|\alpha| \leq N, x \in G} (1 + |x|_\kappa)^{(N+1)(Q+1)} |\widetilde{X}^\alpha \Phi (x)|.
\end{align*}
Moreover, given $N \in \mathbb{N}$, we define the grand maximal function $M_{(N)}f$ of $f \in \mathcal{S}'(G)$ by

\begin{definition}
For $0<p<\infty$, the Hardy space $H^p(G)$ is defined as
\begin{align*}
H^p (G): = \big\{f \in \mathcal{S}'(G) : M_{(N_p)} f <\infty\big\},
\end{align*}
where
$$N_p := \min\{[\alpha]: \alpha \in \mathbb{N}_0^n \mbox{ with } [\alpha] > Q (1/p-1)\}.$$
The quasi-norm on $H^p(G)$ is defined by
\begin{align*}
\|f\|_{H^p (G)}: = \|M_{(N_p)} f\|_{L^p (G)}.
\end{align*}
\end{definition}

\medskip
The Hardy spaces $H^p (G)$, initially defined via grand maximal function, can be
characterized by radial maximal function and nontangential maximal function equivalently. To recall these maximal characterizations,
we need the notion of commutative approximate identities  introduced in \cite{FS}. A commutative approximate identity
on $G$ is a function $\Phi \in \mathcal{S}(G)$ such that
$\int_{G} \Phi (x)dx =1$ and $\Phi_s  \ast \Phi_t =\Phi_t  \ast \Phi_s$ for all $s, t >0$.
On a graded Lie  group it is easy to construct a commutative approximate identity. Indeed,
if $\mathcal{R}$ is a positive Rockland operator on $G$, and $\varphi \in \mathcal{S}(\mathbb{R})$
such that $\varphi(0) =1$, then the convolution kernel of the operator $\varphi(\mathcal{R})$
is a commutative approximate identity.

\begin{proposition} {\rm (\cite[Corollary 4.17]{FS})}
Suppose $0< p <\infty$ and $\Phi$ is a commutative approximate identity. Then for $f \in \mathcal{S}'(G)$, the following are equivalent:
\begin{itemize}
\item[\rm (i)] $M_{\Phi}^0 f \in L^p(G)$;
\item[\rm (ii)] $M_{\Phi} f \in L^p(G)$;
\item[\rm (iii)] $M_{(N_p)} f \in L^p(G)$.
\end{itemize}
Moreover, we have
$$\|M_{\Phi}^0 f\|_{L^p(G)} \sim  \|M_{\Phi}f\|_{L^p(G)} \sim \|M_{(N_p)}f\|_{L^p (G)}$$
with the implicit constants depending only on $\Phi$ and $p$.
\end{proposition}

\begin{remark}
If $1<p <\infty$, the spaces $H^p (G)$ and $L^p(G)$ coincide with equivalent norms. See \cite[p. 75]{FS}.
\end{remark}

\subsection{Atomic decomposition}
Atomic decomposition is a very useful tool for the study of boundedness of operators on Hardy spaces.
Analogously to the Euclidean case, Hardy spaces on graded Lie groups also admit an atomic decomposition,
which we now recall. See \cite{FS} for more details.

A triplet $(p,q,M)$ is said to be admissible,
if $0< p \leq 1 \leq q \leq \infty$, $p \neq q$ and $M \in \mathbb{N}_0$ with $M \geq \max \big\{[\alpha]: \alpha \in \mathbb{N}_0^n
\mbox{ with } [\alpha] \leq Q(1/p-1)\big\}$.
\begin{definition}
 Given an admissible triplet $(p,q,M)$, we say that a function $a$ on $G$ is a
 $(p,q,M)$-atom, if it is a compactly supported $L^q$ function such that
\begin{itemize}
\item[(i)] there is a ball $B$ such that $\supp a \subset \overline{B}$ and $\|a\|_{L^q} \leq \mu(B)^{1/q - 1/p}$;
\item[(ii)]  for every $P\in \mathcal{P}_M$, $\int_G a(x)P(x)d\mu(x) =0$.
\end{itemize}
\end{definition}
The atomic decomposition of $H^p (G)$ can be stated as follows.
\begin{proposition}[\cite{FS}] \label{atomicdec}
Let $(p,q,M)$ be an admissible triplet. Then there is a constant $c_1>0$ such that for all any $(p,q,M)$-atom $a$, one has
\begin{align*}
 \|a\|_{H^p (G)} \leq c_1.
 \end{align*}
Conversely, given any $f \in H^p (G)$,
there exist a sequence $\{a_j\}_{j=1}^{\infty}$ of $(p,q,M)$-atoms and a sequence $\{\lambda_j\}_{j=1}^{\infty}$ of complex numbers such that
$f = \sum_{j=1}^{\infty}\lambda_j a_j$ with convergence in $\mathcal{S}'(G)$ and
\begin{align*}
\left(\sum_{j=1}^{\infty}|\lambda_j|^p\right)^{1/p} \leq c_2 \|f\|_{H^p (G)},
\end{align*}
where $c_2$ is a constant independent of $f$.
\end{proposition}

\section{Proof of main result} \label{proofoftheorem}

We need the following Taylor's formula with integral remainder on homogeneous groups,
due to  Bonfiglioli (see \cite[Theorem 2]{Bon}). Note that
in \cite{Bon} it is assumed that $v_1 =1$, in which case one has $\lceil M \rfloor = M$.

\begin{lemma}\label{TI}
Suppose $f \in C^{M+1}(G)$ for some $M \in \mathbb{N}_0$.
Let $y \mapsto P_{x,M}^f(y)$ denote the right Taylor polynomial of $f$ at $x$ of homogeneous degree $M$.
Then there exists a positive group constant $C_M$ such that
\begin{align*}
f(yx) -P^f_{x,M} (y)& = \sum_{\substack{|\alpha| \leq \lceil M \rfloor \\ [\alpha] \geq M +1}}  \widetilde{X}^\alpha
f (x) \left(\sum_{\beta: [\beta] = [\alpha]} C_{\alpha,\beta} y^\beta \right)\\
& \quad + \sum_{\substack{|\alpha| \leq \lceil M\rfloor +1 \\ [\alpha] \geq M+1}}
\left(\sum_{\beta: [\beta] = [\alpha]} C_{\alpha,\beta}' y^\beta \right)\int_0^1 (\widetilde{X}^\alpha f)(y_{(t)} x) \frac{(1-t)^M}{M!}dt,
\end{align*}
where
\begin{align*}
y_{(t)}:= \exp\left(\sum_{j=1}^n ty_j X_j\right) \equiv (ty_1, \cdots, ty_n),
\end{align*}
$\lceil M\rfloor:= \max\{|\alpha|: \alpha \in \mathbb{N}_0^n \mbox{ with } [\alpha] \leq M\}$,
and $C_{\alpha,\beta}$, $C_{\alpha,\beta}'$ are constants.
\end{lemma}

\begin{lemma} \label{ker42}
Suppose that $\sigma = \{\sigma(\pi), \pi \in \widehat{G}\}$ is a measurable field of operators, $\mathcal{R}$ is positive
Rockland operator (of homogeneous degree $\nu$), and $N$ is an integer,
all of which satisfy the hypothesis of Theorem \ref{main1}. Let $\varphi \in \mathcal{S}(\mathbb{R})$ such that
$\supp \varphi \subset [2^{-\nu}, 2^{\nu}]$ and
\[
\sum_{j \in \mathbb{Z}}\varphi (2^{-\nu j}\lambda) =1 \quad \forall \lambda  \in (0, \infty).
\]
Let $\sigma_j(\pi):= \sigma(2^j \cdot \pi)\varphi(\pi(\mathcal{R}))$ and $K_j := \mathcal{F}_G^{-1} \sigma_j$
for $j \in \mathbb{Z}$. Then
for any $\alpha \in \mathbb{N}_0^n$, there exits a constant $C$ (depending on $\alpha$) such that
\begin{equation} \label{kerest}
\int_G (1 + |x|)^{2N} |\widetilde{X}^\alpha K_j(x)|^2d\mu(x) \leq C.
\end{equation}
\end{lemma}
\begin{proof}
Let $\psi \in \mathcal{S}(\mathbb{R})$ such that $\psi =1$ on $[2^{-\nu}, 2^{\nu}]$. Then $\varphi (\lambda) = \varphi(\lambda) \psi(\lambda)$
for all $\lambda \in \mathbb{R}$. Consequently
\begin{align*}
\sigma_j (\pi) = \sigma_j (\pi) \psi(\pi (\mathcal{R}) ).
\end{align*}
Letting $\Psi$ be the convolution kernel of $\psi(\mathcal{R})$, it follows that
\begin{equation*}
\widetilde{X}^\alpha  K_{j } (x) = \widetilde{X}^\alpha   \big( \Psi  \ast K_j \big) (x)
 =(\widetilde{X}^\alpha \Psi )  \ast K_{j}   (x),
\end{equation*}
where we used \eqref{decon}.
From the quasi-triangle inequality \eqref{quas} we have
\begin{align*}
(1+ |x|)^{N} \lesssim (1 + |xy^{-1}|)^{N} (1 + |y|)^{N},
\end{align*}
which yields
\begin{align*}
(1+ |x|)^{N} \big|(\widetilde{X}^\alpha \Psi)  \ast K_j   (x) \big|
\lesssim   \big[(1+|\cdot|)^{N} |\widetilde{X}^\alpha \Psi|\big]
 \ast \big[(1 + |\cdot|)^{N} |K_j|\big] (x)
\end{align*}
By Hulanicki's theorem (see Corollary \ref{corohul}), we have $\Psi \in \mathcal{S}(G)$, which implies
\[
(1+|\cdot|)^{N} |\widetilde{X}^\alpha \Psi| \in L^1 (G).
\]
Hence by Young's inequality,
\begin{equation} \label{408}
\begin{split}
&\int_G (1 + |x|)^{2N} |\widetilde{X}^\alpha K_j(x)|^2d\mu(x)  \\
& \hspace{2cm}\lesssim \left\|  \big[(1+|\cdot|)^{N} |\widetilde{X}^\alpha \Psi|\big]
 \ast \big[(1 + |\cdot|)^{N} |K_j|\big] \right\|_{L^2(G)}^2\\
 &\hspace{2cm}\lesssim  \int_{G} (1+ |x|)^{2N}|  K_{j}(x)|^2 d\mu (y)\\
& \hspace{2cm}\sim \int_{G} (1+ |x|_\kappa^{2\kappa})^{N/\kappa}|  K_{j}(x)|^2 d\mu (y),
\end{split}
\end{equation}
where $|\cdot|_\kappa$ is the homogeneous quasi-norm defined by \eqref{xkappa}.

Since $N$ is a common multiple of the dilation weights $v_1, \cdots, v_n$ and $\kappa$
is the smallest such common multiple,  $(1+ |x|_\kappa^{2\kappa})^{N/\kappa}$ of the form
\[
(1+ |x|_\kappa^{2\kappa})^{N/\kappa} = \sum_{[\beta] \leq N} c_\beta (x^\beta)^2.
\]
Inserting this into \eqref{408}, and using the Plancherel theorem and the Lebniz rule \eqref{leibniz}, we have
\begin{equation} \label{111111}
\begin{split}
&\int_G (1 + |x|)^{2N} |\widetilde{X}^\alpha K_j(x)|^2d\mu(x)    \\
 & \hspace{2cm}\lesssim   \sum_{[\beta] \leq N} \left\|\Delta^{\beta} \big[\sigma_{j}(\pi)
\varphi(\pi(\mathcal{R})) \big] \right\|_{L^2(\widehat{G})}  \\
& \hspace{2cm} \lesssim \sum_{[\beta'] + [\beta''] \leq N}
\left\|(\Delta^{\beta'} \sigma_{j})(\pi)  \Delta^{\beta''}\big[ \varphi(\pi(\mathcal{R}))\big]\right\|_{L^2(\widehat{G})}.
\end{split}
\end{equation}
Inserting the powers $\pi(\mathcal{R})^{[\beta'] /\nu}$, each term in the above sum
can be estimated as follows:
\begin{equation} \label{222222}
\begin{split}
&\left\|(\Delta^{\beta'} \sigma_{j})(\pi)  \Delta^{\beta''}\big[ \varphi(\pi(\mathcal{R}))\big]\right\|_{L^2(\widehat{G})} \\
&\hspace{2cm} \leq \left\|(\Delta^{\beta'} \sigma_{j})(\pi)  \pi(\mathcal{R})^{[\beta'] /\nu}
\right\|_{L^\infty (\widehat{G})}  \left\| \pi(\mathcal{R})^{-[\beta'] /\nu}\Delta^{\beta''}\big[\varphi(\pi(\mathcal{R}))\big]
\right\|_{L^2(\widehat{G})} \\
&\hspace{2cm}=: E_1 \cdot E_2.
\end{split}
\end{equation}

For the factor $E_1$, we have
\begin{equation} \label{333333}
\begin{split}
E_1 &=\sup_{\pi \in \widehat{G} }\left\|(\Delta^{\beta'} \sigma_{j})(\pi)
\pi (\mathcal{R})^{[\beta'] /\nu}  \right\|_{\mathscr{L}(\mathcal{H}_{\pi})} \\
&= \sup_{\pi \in \widehat{G}}\left\| 2^{j [\beta']} (\Delta^{\beta'} \sigma \big) (2^{j}\cdot\pi)
  \pi(\mathcal{R})^{[\beta'] /\nu}  \right\|_{\mathscr{L}(\mathcal{H}_{\pi})} \\
  &= \sup_{\pi \in \widehat{G}}\left\|   (\Delta^{\beta'} \sigma \big) (2^{j}\cdot\pi)
  \big[(2^j\cdot\pi)(\mathcal{R})\big]^{[\beta'] /\nu}  \right\|_{\mathscr{L}(\mathcal{H}_{\pi})} \\
   &= \sup_{\pi \in \widehat{G}}\left\|   (\Delta^{\beta'} \sigma \big) ( \pi)
  \big[\pi(\mathcal{R})\big]^{[\beta'] /\nu}  \right\|_{\mathscr{L}(\mathcal{H}_{\pi})} \\
  & \leq  C_{\beta'},
\end{split}
\end{equation}
where the last inequality follows from the Mihlin-type condition \eqref{imi1}.

Using Hulanicki's theorem (\ref{corohul}) and the fact that $\varphi$ vanishes
near the origin, one can show that (see the proof of Proposition 4.9 in \cite{FR1} for details)
\begin{equation} \label{444444}
\begin{split}
E_2 =  \left\| \pi(\mathcal{R})^{-[\beta'] /\nu}\Delta^{\beta''}\big[\varphi(\pi(\mathcal{R}))\big]
\right\|_{L^2(\widehat{G})} \leq C_{\beta', \beta''}.
\end{split}
\end{equation}

Combining \eqref{111111} through \eqref{444444} yields the desired estimate \eqref{kerest}.
\end{proof}

Now we give the proof of our main result.

\begin{proof}[Proof of Theorem \ref{main1}]
Let $0< p \leq 1$ and let $\sigma$, $\mathcal{R}$ and $N$ satisfy the hypothesis of Theorem \ref{main1}.
We fix an integer $M$ such that
\begin{align} \label{467}
M \geq \max \big\{[\alpha]: \alpha \in \mathbb{N}_0^n
\mbox{ with } [\alpha] \leq Q(1/p-1)\big\}
\end{align}
and
\begin{align} \label{4688}
\frac{Q}{2}+ (M+1) -N >0.
\end{align}
The condition \eqref{467} means that $(p,2,M)$ is an admissible triplet. Hence,
to prove that $T_\sigma$ is bounded from $H^p(G)$ to $L^p (G)$, by  Proposition \ref{atomicdec} it suffices to show that
there exists a constant $C$ such that for an arbitrary $(p,2,M)$-atom $a$,
\begin{align} \label{desdes}
\|T_\sigma a\|_{L^p (G)} \lesssim 1.
\end{align}

Suppose $a$ is a $(p,2,M)$-atom associated to a ball $B= B(x_0, r)$.
Since $T_\sigma$ commutes with left translations, we may assume without loss of generality
that $x_0 =e$, i.e., the ball $B$ is centered at the group identity $e$. Let $c$ be a sufficient large positive constant.
We write
\begin{align*}
\|T_\sigma a\|_{L^p (G)}^p
&= \int_{B(e, cr)}|T_\sigma a (x)|^p d\mu(x) +   \int_{B(e, cr)^c}|T_\sigma a (x)|^p d\mu(x) \\
&=:I_1 + I_2.
\end{align*}

First we estimate $I_1$. Indeed,
by the Plancherel theorem and H\"{o}lder's inequality,
\begin{align} \label{I1I1}
I_1 \lesssim  \|T_{\sigma} a\|_{L^2 (G)} ^{p}|B(e, c r)|^{1- \frac{p}{2}}
  \lesssim  \|\sigma\|_{L^{\infty}(\widehat{G})}^p \|a\|_{L^2 (G)} ^{p}|B(e, c r)|^{1- \frac{p}{2}}  \lesssim 1.
\end{align}

Next we estimate $I_2$. Choose a function $\varphi \in \mathcal{S}(\mathbb{R})$ such that
$\supp \varphi \subset [2^{-\nu}, 2^{\nu}]$ and
\begin{align*}
\sum_{j \in \mathbb{Z}}\varphi (2^{-\nu j}\lambda) =1 \quad \forall \lambda  \in (0, \infty).
\end{align*}
By the spectral theorem (recalling that
$0$ can be neglected in the spectral resolution),
\begin{align*}
a = \sum_{j \in \mathbb{Z}}\varphi (2^{-\nu j}\mathcal{R})  a \quad \mbox{in } L^2(G).
\end{align*}
Consequently (using the $L^2$-boundedness of $T_\sigma$) we have
\begin{align} \label{lpdecom}
T_\sigma a (x) =  \sum_{j \in \mathbb{Z}} T_\sigma \varphi (2^{-\nu j}\mathcal{R}) a (x), \quad \text{a.e. } x \in G.
\end{align}
We set
\begin{align*}
\sigma_j (\pi)& := \sigma(2^{j}\cdot \pi)\varphi(\pi(\mathcal{R})), \\
\widetilde{\sigma}_j (\pi) &:= \sigma_j (2^{-j}\cdot \pi) = \sigma(\pi)\varphi(2^{-j}\cdot \pi(\mathcal{R})), \\
K_{j}&:= \mathcal{F}_G^{-1}  \sigma_j ,\\
\widetilde{K}_{j}&:= \mathcal{F}_G^{-1}  \widetilde{\sigma}_j .
\end{align*}
Then \eqref{lpdecom} can be rewritten as
\begin{align*}
T_\sigma a (x) = \sum_{j \in \mathbb{Z}} T_{\widetilde{\sigma}_j } a(x)  =\sum_{j \in \mathbb{Z}}\int_{B(e,r)} \widetilde{K}_j
(y^{-1}x)a(y)d\mu(y)=\sum_{j \in \mathbb{Z}} F_j (x),
\end{align*}
where
\[
F_j(x) :  = \int_{B(e,r)} \widetilde{K}_{j}(y^{-1}x)a(y) d\mu(y).
\]
Using $(\sum_j u_j)^p \leq \sum_j |u_j|^p$ ($0< p \leq 1$) and H\"{o}lder's inequality, it follows that
\begin{equation} \label{1111}
\begin{split}
&\int_{B(e, cr)^{c}} |T_\sigma a (x)|^p d\mu(x)
=\int_{B(e, cr)^{c}} \left|\sum_{j \in \mathbb{Z}}F_j (x) \right|^{p} d\mu(x)  \\
&\hspace{1.5cm}\leq \sum_{j \in \mathbb{Z}}\int_{B(e,cr)^c} \left|F_j (x) \right|^p d\mu(x)  \\
&\hspace{1.5cm} = \sum_{j \in \mathbb{Z}}\int_{B(e,cr)^c}|x|^{-pN} |x|^{pN}\left|F_j (x) \right|^p d\mu(x) \\
&\hspace{1.5cm} \leq \sum_{j\in \mathbb{Z}}\left(\int_{B(e,cr)^c} |x|^{-  \frac{2pN}{2-p}} d\mu(x)\right)^{1-\frac{p}{2}}
\left(\int_{B(e,cr)^{c}} |x|^{2N }|F_j (x) |^2 d\mu(x)\right)^{\frac{p}{2}} \\
&\hspace{1.5cm} \sim \sum_{j \in \mathbb{Z}} r^{\frac{(2-p)Q}{2} - pN}\left(\int_{B(e,cr)^c} |x|^{2N }|F_j (x) |^2
d\mu(x)\right)^{\frac{p}{2}}.
\end{split}
\end{equation}
Here we also used the assumption that $N > Q(1/p -1/2)$, which implies $\frac{2pN}{2-p} > Q$ and hence the function $|\cdot|^{\frac{2pN}{2-N}}$
is integrable on $B(e,cr)^c$.

Let $j_0$ be the unique integer such that $2^{-j_0} \leq r <2^{-j_0 +1}$ (i.e., $r \sim 2^{-j_0}$).
To estimate the last integral in  \eqref{1111}, we shall consider two cases: $j \geq j_0$ and $j < j_0$.

{\bf Case 1:} $j \geq j_0$. Observe that  $|x| \sim |y^{-1}x|$ whenever $x \in B(e,cr)^c$ and $y \in B(e,r)$. Hence
for all $x \in B(e,cr)^c$,
\begin{align*}
|x|^{N}|F_j (x)|
& = |x|^{N}\left|\int_{B(e,r)} \widetilde{K}_j (y^{-1}x)a(y) d\mu(y)\right| \\
 & \sim \int_{B(e,r)}
  |y^{-1}x|^{N} \widetilde{K}_j (y^{-1}x) ||a(y)|d\mu(y) \\
&\leq \|a\|_{L^2(G)} \left(\int_{B(e,r)} |y^{-1}x|^{2N}| \widetilde{K}_j (y^{-1}x)|^2 d\mu(y) \right)^{\frac{1}{2}} \\
&\lesssim   |B(e,r)|^{\frac{1}{2} - \frac{1}{p}} \left(\int_{B(e,r)} |y^{-1}x|^{2N}| \widetilde{K}_j (y^{-1}x)|^2 d\mu(y) \right)^{\frac{1}{2}}.
\end{align*}
It follows by Fubini's theorem, the fact that $\widetilde{K}_j (x) = 2^{jQ}K_j(\delta_{2^j}x)$, and Lemma \ref{ker42}  that
 \begin{equation} \label{ecase1}
 \begin{split}
&\int_{B(e,cr)^{c}} |x|^{2N }|F_j (x) |^2 d\mu(x) \\
&\hspace{1cm} \lesssim  |B(e,r)|^{1 - \frac{2}{p}} \int_{B(e,cr)^{c}}
\left(\int_{B(e,r)} |y^{-1}x|^{2N}| \widetilde{K}_j (y^{-1}x)|^2 d\mu(y) \right) d\mu(x)  \\
&\hspace{1cm}=   |B(e,r)|^{1 - \frac{2}{p}} \int_{B(e,r)} \left(\int_{B(e,cr)^{c}}
  |y^{-1}x|^{2N}| \widetilde{K}_j (y^{-1}x)|^2   d\mu(x)\right)d\mu(y)   \\
  &\hspace{1cm}\leq   |B(e,r)|^{1 - \frac{2}{p}} \int_{B(e,r)} \left(\int_{G}
  |y^{-1}x|^{2N}| \widetilde{K}_j (y^{-1}x)|^2   d\mu(x)\right)d\mu(y)\\
    &\hspace{1cm}=  |B(e,r)|^{1 - \frac{2}{p}} \int_{B(e,r)} \left(\int_{G}
  |x|^{2N}| \widetilde{K}_j (x)|^2   d\mu(x)\right)d\mu(y)   \\
   &\hspace{1cm}=  |B(e,r)|^{2 - \frac{2}{p}} \int_{G}
  |x|^{2N}| \widetilde{K}_j (x)|^2   d\mu(x)\\
  & \hspace{1cm}=  |B(e,r)|^{2 - \frac{2}{p}}2^{j(Q-2N)} \int_{G}
  |x|^{2N} | K_j ( x)|^2   d\mu(x)\\
  &\hspace{1cm}\lesssim r^{(2-\frac{2}{p})Q}2^{j(Q -2N)}.
\end{split}
\end{equation}

{\bf Case 2:} $j < j_0$. Let $P^{\widetilde{K}_j}_{x, M}$ be the right Taylor polynomial of $\widetilde{K}_j$
at $x$ of homogeneous degree $M$. By the vanishing moments of $a$, we have, for each $x \in B(e,cr)^c$,
\begin{align} \label{66a}
 F_j (x)
 = \int_{B(e,r)} \left[\widetilde{K}_j (y^{-1}x) -P^{\widetilde{K}_j}_{x, M}(y^{-1})\right]a(y) d\mu(y)
\end{align}
Using Lemma \ref{TI} we can write
\begin{equation}\label{66b}
\begin{split}
\widetilde{K}_j (y^{-1}x) -P^{\widetilde{K}_j}_{x, M}(y^{-1}) &=
\sum_{\substack{|\alpha| \leq \lceil M \rfloor \\ [\alpha] \geq M +1}}  \widetilde{X}^\alpha
\widetilde{K}_j (x) P_\alpha (y)\\
& \quad + \sum_{\substack{|\alpha| \leq \lceil M\rfloor +1 \\ [\alpha] \geq M+1}}
P_\alpha'(y)\int_0^1 (\widetilde{X}^\alpha \widetilde{K}_j)\big((y^{-1})_{(t)} x\big) \frac{(1-t)^M}{M!}dt,
\end{split}
\end{equation}
where both $P_\alpha$ and $P_\alpha '$ are polynomials on $G$ of homogeneous degree $[\alpha]$, and
\[
(y^{-1})_{(t)}:= \exp \left(- \sum_{j=1}^n y_j X_j \right) \equiv (-ty_1, \cdots, -ty_n).
\]
Inserting \eqref{66b} into \eqref{66a}, and using the Cauchy-Schwarz inequality and the size condition of atoms, we have
\begin{align*}
|F_j(x)| &\leq \sum_{\substack{|\alpha| \leq \lceil M\rfloor  \\ [\alpha] \geq M+1}}
\int_{B(e,r)}  \left|\widetilde{X}^\alpha
\widetilde{K}_j (x) P_\alpha (y)a(y)\right| d\mu(y) \\
& \quad +\sum_{\substack{|\alpha| \leq \lceil M\rfloor +1 \\ [\alpha] \geq M+1}}
\int_{B(e,r)}  \left|P_\alpha'(y)\int_0^1 (\widetilde{X}^\alpha \widetilde{K}_j)\big((y^{-1})_{(t)} x\big) \frac{(1-t)^M}{M!}dt\right| |a(y)|d\mu(y) \\
& \lesssim \sum_{\substack{|\alpha| \leq \lceil M\rfloor  \\ [\alpha] \geq M+1}}r^{(\frac{1}{2}-\frac{1}{p})Q}
\left(\int_{B(e,r)}   \left|\widetilde{X}^\alpha
\widetilde{K}_j (x) P_\alpha (y)\right|^2 d\mu(y)\right)^{1/2} \\
& \quad +\sum_{\substack{|\alpha| \leq \lceil M\rfloor +1 \\ [\alpha] \geq M+1}}r^{(\frac{1}{2}-\frac{1}{p})Q}
\left(\int_{B(e,r)}  \left|\int_0^1 P_\alpha'(y)(\widetilde{X}^\alpha \widetilde{K}_j)\big((y^{-1})_{(t)} x\big) \frac{(1-t)^M}{M!}dt\right|^2 d\mu(y)\right)^{1/2} \\
& =: F^{(1)}_j(x) + F^{(2)}_j(x).
\end{align*}
Thus,
\begin{align*}
\int_{B(e,cr)^{c}} |x|^{2N }|F_j (x) |^2 d\mu(x) \lesssim \int_{B(e,cr)^{c}} |x|^{2N }|F_j^{(1)} (x) |^2 d\mu(x)
 + \int_{B(e,cr)^{c}} |x|^{2N }|F_j^{(2)} (x) |^2 d\mu(x).
\end{align*}

We first estimate $\int_{B(e,cr)^{c}} |x|^{2N }|F_j^{(2)} (x) |^2 d\mu(x)$. Indeed, by Minkowski's
inequality and Fubini's theorem,
{\small \begin{align*}
&\int_{B(e,cr)^{c}} |x|^{2N }|F_j^{(2)} (x) |^2 d\mu(x) \\
& \lesssim  \sum_{\substack{|\alpha| \leq \lceil M\rfloor +1 \\ [\alpha] \geq M+1}} r^{(1-\frac{2}{p})Q}
\int_{B(e,cr)^{c}} |x|^{2N }\left(\int_{B(e,r)}  \left|\int_0^1 P_\alpha'(y)(\widetilde{X}^\alpha \widetilde{K}_j)\big((y^{-1})_{(t)} x\big) \frac{(1-t)^M}{M!}dt\right|^2 d\mu(y)\right) d\mu(x)\\
& \leq  \sum_{\substack{|\alpha| \leq \lceil M\rfloor +1 \\ [\alpha] \geq M+1}} r^{(1-\frac{2}{p})Q}
\int_{B(e,cr)^{c}} |x|^{2N }\left[\int_0^1 \left(\int_{B(e,r)}  \left|P_\alpha'(y)(\widetilde{X}^\alpha \widetilde{K}_j)\big((y^{-1})_{(t)} x\big)
\right|^2 d\mu(y)\right)^{1/2}dt\right]^2 d\mu(x)\\
& \leq  \sum_{\substack{|\alpha| \leq \lceil M\rfloor +1 \\ [\alpha] \geq M+1}} r^{(1-\frac{2}{p})Q}
 \left[\int_0^1 \left(\int_{B(e,cr)^{c}}|x|^{2N } \int_{B(e,r)}  \left|P_\alpha'(y)(\widetilde{X}^\alpha \widetilde{K}_j)\big((y^{-1})_{(t)} x\big)
\right|^2 d\mu(y)d\mu(x)\right)^{1/2}dt\right]^2 \\
& =  \sum_{\substack{|\alpha| \leq \lceil M\rfloor +1 \\ [\alpha] \geq M+1}} r^{(1-\frac{2}{p})Q}
 \left[\int_0^1 \left(\int_{B(e,r)} |P_\alpha'(y)|^2\int_{B(e,cr)^c} |x|^{2N } \left|(\widetilde{X}^\alpha \widetilde{K}_j)\big((y^{-1})_{(t)} x\big)
\right|^2 d\mu(x)d\mu(y)\right)^{1/2}dt\right]^2 .
\end{align*}
Note} that if $y \in B(e,r)$, $x \in B(e,cr)^c$ and $t \in [0,1]$, then $|(y^{-1})_{(t)}x| \sim |x|$. Thus, for every $y \in B(e,r)$
and $t \in [0,1]$,
\begin{align*}
&\int_{B(e,cr)^c} |x|^{2N } \left|(\widetilde{X}^\alpha \widetilde{K}_j)\big((y^{-1})_{(t)} x\big)
\right|^2 d\mu(x) \\
&\quad \sim \int_{B(e,cr)^c} |(y^{-1})_{(t)}x|^{2N } \big|(\widetilde{X}^\alpha \widetilde{K}_j)\big((y^{-1})_{(t)} x\big)
\big|^2 d\mu(x)  \\
&\quad\leq \int_G  |x|^{2N } \big|(\widetilde{X}^\alpha \widetilde{K}_j)(x)
\big|^2 d\mu(x) \\
&\quad = 2^{j(Q + 2[\alpha] -2N)}\int_G  |x|^{2N } \big|(\widetilde{X}^\alpha {K}_j)(x)
\big|^2 d\mu(x) \\
& \quad \lesssim  2^{j(Q + 2[\alpha] -2N)},
\end{align*}
where for the last line we used Lemma \ref{ker42}.
Inserting this estimate yields
\begin{align*}
&\int_{B(e,cr)^{c}} |x|^{2N }|F_j^{(2)} (x) |^2 d\mu(x)  \\
& \quad\lesssim  \sum_{\substack{|\alpha| \leq \lceil M\rfloor +1 \\ [\alpha] \geq M+1}} r^{(1-\frac{2}{p})Q} 2^{j(Q + 2[\alpha] -2N)}
 \left[\int_0^1 \left(\int_{B(e,r)} |P_\alpha'(y)|^2d\mu(y)\right)^{1/2}dt\right]^2 \\
&\quad \leq  \sum_{\substack{|\alpha| \leq \lceil M\rfloor +1 \\ [\alpha] \geq M+1}} r^{(1-\frac{2}{p})Q}
 2^{j(Q + 2[\alpha] -2N)} \int_{B(e,r)} |y|^{2[\alpha]}d\mu(y) \\
&\quad \sim \sum_{\substack{|\alpha| \leq \lceil M\rfloor +1 \\ [\alpha] \geq M+1}} r^{(2-\frac{2}{p})Q + 2[\alpha]} 2^{j(Q + 2[\alpha] -2N)}.
\end{align*}

With a similar but easier argument, we get the analogous estimate
\begin{align*}
\int_{B(e,cr)^{c}} |x|^{2N }|F_j^{(1)} (x) |^2 d\mu(x)
\lesssim \sum_{\substack{|\alpha| \leq \lceil M\rfloor \\ [\alpha] \geq M+1}}  r^{(2-\frac{2}{p})Q + 2[\alpha]} 2^{j(Q + 2[\alpha] -2N)}.
\end{align*}
The details are left to the interested reader.

Therefore, we have
\begin{align} \label{ecase2}
\int_{B(e,cr)^{c}} |x|^{2N }|F_j  (x) |^2 d\mu(x)
\lesssim \sum_{\substack{|\alpha| \leq \lceil M\rfloor  +1\\ [\alpha] \geq M+1}}  r^{(2-\frac{2}{p})Q + 2[\alpha]} 2^{j(Q + 2[\alpha] -2N)}.
\end{align}

Inserting \eqref{ecase1} and \eqref{ecase2} into \eqref{1111}, and using that $r \sim 2^{-j_0}$, we obtain
\begin{align*}
&\int_{B(e, cr)^{c}} |T_\sigma a (x)|^p d\mu(x) \\
&\quad\lesssim  \sum_{j \in \mathbb{Z}} r^{\frac{(2-p)Q}{2} - pN}\left(\int_{B(e,cr)^c} |x|^{2N }|F_j (x) |^2
d\mu(x)\right)^{\frac{p}{2}} \\
&\quad\lesssim \sum_{j \geq j_0} r^{\frac{(2-p)Q}{2} - pN}\big[r^{(2-\frac{2}{p})Q}2^{j(Q -2N)}\big]^{\frac{p}{2}}
 +   \sum_{j < j_0}\sum_{\substack{|\alpha| \leq \lceil M\rfloor +1 \\ [\alpha] \geq M+1}}r^{\frac{(2-p)Q}{2} - pN}
 \big[ r^{(2-\frac{2}{p})Q + 2[\alpha]} 2^{j(Q + 2[\alpha] -2N)}\big]^{\frac{p}{2}} \\
&\quad \sim \sum_{j \geq j_0} 2^{(j-j_0)p(\frac{Q}{2}-N)}  +
 \sum_{j < j_0}\sum_{\substack{|\alpha| \leq \lceil M\rfloor +1 \\ [\alpha] \geq M+1}}2^{(j-j_0)p(\frac{Q}{2}+ [\alpha] -N)}.
\end{align*}
Since $N > Q(1/p -1/2) > Q/2$, we have
\[
\sum_{j \geq j_0} 2^{(j-j_0)p(Q/2-N)} \lesssim 1.
\]
Note that the index set $\mathcal{I}:= \{\alpha  \in \mathbb{N}_0^n :|\alpha| \leq \lceil M\rfloor +1, \ [\alpha] \geq M+1\}$ has finite elements.
Moreover, the condition \eqref{4688} implies that $\frac{Q}{2}+ [\alpha] -N >0$ whenever $\alpha \in \mathcal{I}$. Thus
\begin{align*}
\sum_{j < j_0}\sum_{\substack{|\alpha| \leq \lceil M\rfloor +1 \\ [\alpha] \geq M+1}}2^{(j-j_0)p(\frac{Q}{2}+ [\alpha] -N)}\lesssim 1.
\end{align*}
Therefore,
\begin{align} \label{esofI2}
I_2 =\int_{B(e, cr)^{c}} |T_\sigma a (x)|^p d\mu(x) \lesssim 1,
\end{align}

Combining \eqref{I1I1} and \eqref{esofI2}, we obtain
\[
\|T_\sigma a\|_{L^2(G)}\lesssim 1.
\]
This completes the proof of Theorem \ref{main1}.
\end{proof}

\section*{Acknowledgement}
The authors would like to thank the anonymous referee for the valuable
comments which have helped to improve the presentation of the manuscript.

\end{document}